\documentclass{amsart}
\usepackage{amsfonts,amssymb,amsmath,amsthm, latexsym}
\usepackage{url}
\usepackage{enumerate}
\usepackage{graphicx}
\usepackage{color}

\newtheorem{theorem}{Theorem}[section]
\newtheorem{lemma}[theorem]{Lemma}

\newtheorem{corollary}[theorem]{Corollary}
\theoremstyle{definition}

\numberwithin{equation}{section}

\author{Ethan M. Coven and Reem Yassawi}
\address{Department of Mathematics, Wesleyan University, USA and \\
Department of Mathematics, Trent University, Canada
}
\email{ecoven@wesleyan.edu\\ryassawi@trentu.ca}


\usepackage{amsthm,amssymb,amsmath,latexsym}

\newtheorem{example}[theorem]{Example}

\begin{document}
\title{Endomorphisms and automorphisms of minimal symbolic systems with sublinear complexity}

\maketitle

\begin{abstract}
We show that if the complexity difference function   $p(n+1)-p(n)$ of a infinite minimal shift is bounded, then the the automorphism group of the one-sided shift is finite, and the  automorphism group 
 of the corresponding  two-sided shift ``modulo the shift" is finite.  For the minimal Sturmian and minimal substitution shifts, the bounds  can be explicitly computed, and for linearly recurrent shifts, the bound can be expressed as a function of the linear recurrence constant. We also show   that any endomorphism of a linearly recurrent shift is a root of a power of the shift map.
\end{abstract}

\section{Introduction}

In this note, we use a result by Cassaigne \cite{cas} to obtain a simple proof of the fact
 that any automorphism $\Phi$ of a one- or two-sided minimal shift $(X,\sigma)$ with sublinear complexity  satisfies $\Phi^k=\sigma^n$ for some $k\geq 1$ and $n\in \mathbb Z$.  If $X$ is one-sided, we have $n=0$. (All terms are defined in Section \ref{notation}.)  
 In  Theorem \ref{main_result}, we give bounds on the values that $k$ can assume and in some cases describe how to compute these bounds explicitly. 
 We also   describe endomorphisms of a one-sided linearly recurrent shift in  Theorem \ref{one_sided_coalescence}.

Earlier   results for families of  ``small" shifts include \cite{coven_automorphism}, where the endomorphism monoid of $(X,\sigma)$ is explicitly described for any nontrivial constant length substitution shift on two letters,  and \cite{host_parreau}, where measurable factor maps between constant length substitution shifts which do not have a purely discrete spectrum are characterized.  More recently, any endomorphism of a Sturmian shift is shown to be a power of the shift  in \cite{olli}, the set of factor maps between pairs of constant length or Pisot substitution shifts with the same Perron value  are shown in \cite{salo_torma} to be generated by a finite set, and by the result of \cite{cyr_kra}, we see that modulo the subgroup generated by the shift,  every automorphism of a transitive shift with subquadratic complexity has finite index. All these results are for two-sided shifts.

We recently became aware of \cite{durand_petite_maass} and \cite{cyr_kra_2}, which contain similar  results for the two-sided case. Though there seem to be common threads in all three approaches, the directions taken are different.

\section{Notation} \label{notation}
Let $\mathcal A$ be a finite alphabet, with the discrete topology, and let $\mathbb N_0 = \{0,1,2,\ldots \}$.
  We endow $\mathcal A^{\mathbb N_0}$ and  $\mathcal A^\mathbb Z$ with the product topology, and let $\sigma:\mathcal A^{\mathbb N_0} \rightarrow \mathcal A^{\mathbb N_0}$ (or  $\sigma:\mathcal A^\mathbb Z \rightarrow \mathcal A^\mathbb Z$) denote the shift map.
We consider only minimal and infinite $(X,\sigma)$, which can be either one or two-sided.
 An {\em endomorphism} of $(X, \sigma) $ is a map $\Phi:X\rightarrow X$ which is continuous, onto, and commutes with $\sigma$; if in addition  $\Phi$ is one-to-one, then $\Phi$ is called an {\em automorphism}.  We say $\Phi$ has {\em finite order m} if $m \geq 1$ is the least  number such that
$\Phi^m$ is the identity, and that $\Phi$ is a {\em $k$-th root of a power of the  shift} if there exists an $n$ such that $\Phi^k=\sigma^n$. Let ${\text {Aut}}(X,\sigma)$ denote the automorphism group of $( X,\sigma)$. If $X$ is two-sided,  then ${\text {Aut}}(X,\sigma)$ contains the (normal) subgroup generated by the shift, which we denote by $\{\bar \sigma^n\}$.

  The {\em language} of a minimal shift $(X,\sigma)$, denoted  $\mathcal L_X$, is the set of all finite words that we see in points in $X$. 
  We denote by $p(n)$, $n \ge 1$, the {\em complexity function\/} of
$(X,\sigma)$; $p(n)$ is the number of words in $\mathcal L_X$ of length $n$.
A symbolic system $(X,\sigma)$ has {\em sublinear complexity}
if its complexity function is bounded by a linear function.

  If $(X,\sigma)$ is a one-sided minimal shift, let $(\bar X, \bar\sigma)$ denote the two-sided version of $(X, \sigma)$ i.e. $\bar X$ consists of all bi-infinite sequences such that $\mathcal L_X = \mathcal L_{\bar X}$. Given a two-sided shift $(\bar{X}, \bar\sigma)$ we define similarly its one-sided version.
  If $(X,\sigma)$  is minimal, then the one-sided version of $(\bar X, \bar\sigma)$  is  $(X,\sigma)$  itself.
    {\em Henceforth $(X,\sigma)$ refers to a one-sided shift and $(\bar X, \bar\sigma)$ refers to a two-sided shift.}    
We say that $x\in X$ is a {\em branch point of order $k>1$}  if $|\sigma^{-1}(x)|=k$. We say that
$\bar x= \ldots \bar x_{-1}\bar x_0 \bar x_1 \ldots$ and 
$\bar y= \ldots \bar y_{-1}\bar y_0 \bar y_1 \ldots $
  in $\bar X$ are {\em right asymptotic} if there is some $N\in \mathbb Z$ such that $\bar x_n=\bar y_n$ for $n\geq N$. 
  
 Two orbits $\mathcal O_{\bar x}=\{\bar\sigma^n(x): n\in \mathbb Z\}$ and $\mathcal O_{\bar y}=\{\bar\sigma^n(y): n\in \mathbb Z\}$  in $\bar X$ are {\em right asymptotic } if   there exist $\bar x'\in {\mathcal O}_{\bar x}$ and $\bar y'\in {\mathcal O}_{\bar y}$  that are right asymptotic. A single orbit ${\mathcal O}_{\bar x}$ is {\em right asymptotic\/} if there is 
  a point ~$\bar y$ 
with ${\mathcal O}_{\bar x} \cap {\mathcal O}_{\bar y} = \emptyset$, and ${\mathcal O}_{\bar x}$ and  ${\mathcal O}_{\bar y}$ are right asymptotic.

Define the equivalence relation $\sim$ on the set of asymptotic orbits of  $(\bar X,\bar\sigma)$ by   $\mathcal O_{\bar x}
 \sim
 \mathcal O_{\bar y}$  
 if $\mathcal O_{\bar x}$
and 
$ \mathcal O_{\bar y}$  
 are right asymptotic.

 \section{results}
 Our main result is
 
  \begin{theorem} \label{main_result} 
Let $(X,\sigma)$ be minimal and infinite with sublinear complexity.  Then 
\begin{enumerate}
\item
 if  $X$ has $M_k$  branch points of order $k$ in $X$,  then ${\text {Aut}}(X,\sigma)$ has at most $M :=\min_k M_k$ elements.
 As a consequence any $\Phi\in {\text {Aut}}(X,\sigma)$ has order at most $M$. 
 \item  if   
 for each $k \ge 2$, $\bar X$ has $\bar M_k$ 
 $\sim$-equivalence classes of size $k$,  and $\bar M:= \min_k \bar M_k$,
 then ${\text {Aut}}(\bar X,\bar\sigma)/\{ \bar\sigma^n\}$ has at most $\bar M $ elements. As a consequence any $\Phi\in {\text {Aut}}(\bar X,\bar \sigma)$ is at most an $\bar M$-root of a power of the shift.
\end{enumerate}
 \end{theorem}

\begin{proof}
We shall use Lemmas \ref{cassaigne} and \ref{tool}  below in our proof.

 Using Lemma \ref{cassaigne}, the set $E$ of branch points in $X$ and the collection $\mathcal E$ of right-asymptotic orbits in $\bar X$ are finite. The minimality of our systems implies that an automorphism is determined by its action on any  one  point, and the fact that our systems are infinite implies that $X$, resp. $\bar X$ has at least one branch point, resp. right asymptotic orbit.

An  automorphism of a one-sided system must map a branch point of order $k$ to a  branch point of order $k$, so we apply Lemma \ref{tool}(1)  to the  set of  branch points  of order $K$, where 
where $M_K$ is a minimum of $\{ M_k: k\geq 2\}$.  
 
 An automorphism $\bar \Phi$ of a two-sided system must map a  
$\sim$-equivalence class of size $k$ to a  $\sim$-equivalence class of size $k$, so we apply  Lemma \ref{tool} (2) to the set 
$ \mathcal E$ of  orbits belonging to any $\sim$-equivalence class of size $K$, where
 $\bar M_K$ is a minimum of $\{ \bar M_k: k\geq 2\}$.

Note that if $\bar \Phi$ sends an equivalence class to itself, then $\bar\Phi$ must be a power of the  shift. For if $\bar x_i  =\bar y_i$ for $i\geq 0$ and 
$\bar\Phi$ maps the $\bar\sigma$-orbit of $\bar x$ to the $\bar\sigma$-orbit of $\bar y$,  then for some  $n$ it maps  
$\bar x_{[0,\infty)} = \bar y_{[0,\infty)}$ to $\bar y_{[n,\infty)}$.   Minimality now implies that 
$\bar\Phi = \bar\sigma^n$.  This argument also implies that if $\Phi_1 (x) =y$ and $\Phi_2(x)=y'$ where $y\sim y'$, then $\Phi_1= \Phi_2 \circ \sigma^n$ for some integer $n$.    Thus there are at most $\bar M_K$ automorphisms in ${\text {Aut}}(\bar X,\bar\sigma)/\{ \bar\sigma^n\}$.

 \end{proof}

 As $\text{Aut} (X,\sigma)$ and  $\text{Aut} (\bar X,\bar \sigma)/\{ \sigma^n\}$ are both (finite) groups, then the order of any element of $\text{Aut} (X,\sigma)$ divides $|\text{Aut} (X,\bar\sigma)|$, and any element  of  
 $\text{Aut} (\bar X,\bar \sigma)$ is a $k$-th root of the shift, where $k$ divides   $|\text{Aut} (\bar X,\bar \sigma)/\{\bar \sigma^n\}|$.
 
  Let $s(n):=p(n+1)-p(n)$ be the complexity difference function.
 
 \begin{lemma}\label{cassaigne} Let $(X,\sigma)$ be minimal and infinite. Then (1) is equivalent to  (2), (2) implies (3), and (3) is equivalent to (4).
 \begin{enumerate}
 \item
 $(X,\sigma)$ has sublinear complexity.
  \item
 There exists a constant $L$ such that $s(n)\leq L$ for all $n\geq 1$.
  \item
 $(X,\sigma)$ has finitely branch points. 
 \item
$(\bar X, \bar\sigma)$  has finitely many asymptotic orbits. 
 
 \end{enumerate}
 \end{lemma}
 
 \begin{proof} 
 The only difficult implication is that sublinear complexity implies that the difference function $s(n)$ is bounded, and this is proved in \cite{cas}.  
 Conversely if $s(n)\leq L$ for all $n$, then $p(n)\leq Kn$ for all n, where $K:=\max\{L, p(1)\}$. 
 
 To see that (2) implies (3), first note that $L$ is an upper bound for the number of words of length $n$ which can be extended to the left in at least two ways. If $x$ is a branch point in  $X$, then for each $n$, the prefix of $x$ of length $n$ can be extended in at least two ways. If there were more than $L$ branch points,  choose $n$ large enough so that the prefixes of these branch points are  distinct, a contradiction.

Finally note that the sums of the orders of all the branch points in $(X,\sigma)$ is an upper bound (possibly strict) for the number of right  asymptotic orbits in $(\bar X, \bar\sigma)$.  This shows that (3) is equivalent to (4).
 \end{proof}

 The proof of the following lemma  is straightforward. It follows directly from the minimality of our systems, and will be our main tool.  Here our systems need not be symbolic,  nor invertible in (1), and 
  we will not use bars to indicate invertibility.

 \begin{lemma}\label{tool}
  Let  $(X,S)$ and $(Y,T)$ be infinite  minimal systems. 
\begin{enumerate}
\item  
If there exist  finite sets  $E\subset X$ and $F\subset Y$   such that  $\Phi(E) \subseteq F$ for any
isomorphism $\Phi :(X,S) \to (Y,T)$,
then there are at most $|E|$ isomorphisms from $(X,S)$ to $(Y,T)$. 
\item  Let $ S$ and $T$ be invertible. If there exist finite collections $\mathcal E$ of $S$-orbits and $\mathcal F$ of  $T$-orbits  such that any isomorphism $\Phi: (X,S) \to ( Y,T)$ satisfies $\Phi(\mathcal E)\subset \mathcal F$, then there are at most  $|\mathcal E|$ isomorphisms $ \{\Phi_i: i \in I \}$ such that any isomorphism from $( X,S)$ to $(Y,T)$ is of the form
$\Phi=\Phi_i \circ S^n$ for some $n\in \mathbb Z$ and some $i $. 
\end{enumerate}
\end{lemma}

Note that we can also use  Lemma \ref{tool} to bound the size of the  set of isomorphisms $\Phi :(X,S) \to (Y,T)$ between two minimal systems ``modulo powers of $S$".

\section{Examples and bounds}\label{examples_and_bounds}

We now describe some families of shifts which satisfy the conditions of Theorem \ref{main_result}. 
\begin{itemize}

\item

{\em Sturmian shifts.}
All infinite minimal Sturmian shifts have complexity function $p(n)=n+1$.  This result is implicit in \cite{morse_hedlund}, although \cite{morse_hedlund} does not ever mention blocks.  For details of why $p(n)=n+1$, see \cite{coven_hedlund}.
Since one-sided, resp. two-sided Sturmian shifts have only one branch point, resp.  pair of positively asymptotic orbits, 
from Theorem \ref{main_result} we get

\begin{corollary}  
For
infinite minimal
Sturmian  systems, both 
${\text {Aut}}(X,\sigma)$ and ${\text {Aut}}(\bar X,\bar\sigma)/\{\bar\sigma^n\} $  consist of only the identity map.
\end{corollary}

This  has been proved by J. Olli \cite{olli} using different methods.

\item {\em Linearly recurrent shifts.}
The notions below are the same for one-sided and two-sided shifts, so we will state them only for one-sided shifts, thus avoiding the bars.
If $u,w \in \mathcal L_X$, we that $w$ is a {\em return word to $u$} if (a) $u$ is a prefix of ~$w$,  (b) $wu \in \mathcal L_X$, and (c) there are exactly two occurrences of $u$ in  $wu$. Letting  $\ell(u)$ denote the length of $u$, a minimal one-sided shift $(X,\sigma)$ is {\em linearly recurrent}
if there exists a  constant $K$ such that for any word $u \in \mathcal L_X$
  and any return word (to $u$) $w$,  $\ell(w)\leq K\ell(u)$.  Such a ~$K$ is called a {\em recurrence constant.}
By  \cite[Theorem~23]{dhs}, if $(X,\sigma)$ has linear recurrence constant $K$, then its complexity is bounded above by $Kn$ for large $n$. Cassaigne \cite{cas} shows that if 
$p(n)\leq Kn+1$, then  $s(n) \leq 2K(2K+1)^2$ for all large $n$.
It then follows from Theorem \ref{main_result} that 
\begin{corollary}\label{linearly_recurrent}
Let $(X,\sigma)$ be a linearly recurrent shift  with recurrence constant $K$. Then 
$|{\text {Aut}}(X,\sigma)|\leq    2(K+1)(2K+3)^2 $  and 
$|{\text {Aut}}(\bar X,\bar \sigma)/\{ \bar\sigma^n\}|\leq   2(K+1)(2K+3)^2   $.

\end{corollary}

Note that for linearly recurrent two-sided shifts, any endomorphism is an automorphism (\cite[Corollary~18]{durand_lr}). The appropriate one-sided version of coalescence - when any endomorphism is an automorphism - is then
\begin{theorem}\label{one_sided_coalescence}
 Let $(X,\sigma)$ be a one-sided  linearly recurrent shift  with recurrence constant $K$.  Then every endomorphism of $(X, \sigma )$ is a $k$-th root of a power of the shift for some  positive 
 $ k\leq   2(K+1)(2K+3)^2$.
\end{theorem}
 \begin{proof}
Any endomorphism $\Phi$ of  $(X,\sigma)$ defines an endomorphism $\bar \Phi$ of $(\bar X, \bar\sigma)$, which must be an automorphism by 
\cite[Corollary~18]{durand_lr}. By Corollary \ref{linearly_recurrent}, $\bar \Phi^k =\bar\sigma^n$ for some positive $k \leq    2(K+1)(2K+3)^2     $ and some $n\in \mathbb Z$, so 
$ \Phi^k =\sigma^n$.
\end{proof}

This last result cannot be improved: Example \ref{hedlund} tells us that $\Phi$ is not necessarily  a power of the shift.

\item {\em  Substitution shifts.}
A more tractable subclass of the linearly recurrent shifts are the
primitive substitution shifts: see \cite{BDH} for definitions, where a primitive substitution on $k$ letters is shown to have  at most $k^2$ right asymptotic orbits. We can then deduce

\begin{corollary}\label{global_bound}
Let $(\bar X, \bar\sigma)$ be the minimal shift generated by
a primitive substitution on $k$ letters.
Then any automorphism of $(\bar X, \bar\sigma)$ is an at most $k^2$-th root of the shift.    
\end{corollary}

A substitution $\theta$ is (one-sided) {\em recognizable} \cite{mosse} if any one sided point can,  except for  possibly some initial segment, be uniquely de-substituted. 
 There exist algorithms for finding branch points of recognizable
  primitive substitutions, and hence right asymptotic orbits of  primitive substitution systems. 
 For these substitutions $\theta$, branch points   can arise in one of  two ways.  
 A branch point can be a $\theta$-periodic point satisfying $\theta^n(y)=y$, where there is more than one  word $ay_0\in \mathcal L_X$ such that  $a$ is a suffix of $\theta^n (a)$.
 The only other way a branch point $y$ can arise is when we see maximal proper common suffixes of at least two substitution words $\theta^n(a) $;  in this case $y$ satisfies $w\theta^n(y)= y$ for some $w \in \mathcal L_X$. In both cases such branch points can be listed.

\begin{example}Take the substitution $\theta$ on the alphabet $\{a,b,c\}$ defined by
 $\theta(a)=acb$, $\theta(b)=aba$ and $\theta(c)=aca$. Then 
the right $\theta$-fixed point $u$ is a 2-branch point. 
Note that $\theta(b)$ and $\theta (c)$ both have $a$ as a maximal proper common suffix, and
$\theta^n(b)$ and $\theta^n (c)$ both have $a\theta(a) \ldots \theta^{n-1}(a)$ as a maximal common suffix. Letting $n \rightarrow \infty$, we see that $y =     a \, \theta(a)\,  \theta^2(a) \ldots $ is a 2-branch point and that it  satisfies the equation $a\theta(y) =  y$. There are no other branch points.
According to our bounds, $|\text{Aut} ( X, \sigma)|\leq 2$.

With a little extra work, 
we can see that $\text{Aut}( X, \sigma)$ is trivial.  
Note that $ u \in \theta( X)$ and, since $a \theta (y) = y$, then  $ y \in \sigma^2(\theta( X))$. The  recognizability of $\theta$ tells us that  each point  $ x \in  X$ belongs to exactly one of $ \sigma^{i} \theta(X)$, $i=0,1,2$.
If $\Phi(u)=y$, $\Phi (\theta(X) )= \sigma^2 (\theta(X))$ and $\Phi (\sigma^2 (\theta(X) )= \sigma (\theta(X))$, which implies that $\Phi(y)\neq u$, a contradiction.
\end{example}

The following example is a modification of the example in  \cite[page 372]{h} and shows that not every endomorphism of a one sided shift is a power of the shift.

\begin{example}\label{hedlund}
Let  $B=1001$  and  $C=1101$  and let  $X_0$ be the set of  all concatenations of  B  and  C, where $B$ and $C$ occur starting at multiples of $4$,  that ``mirror"
 points in the Morse-Thue Minimal System (in the sense of \cite{coven_keane_lemasurier}). $(X_0,\sigma^4)$ is isomorphic to the one-sided Morse-Thue system.
 Define
$   X = X_0 \cup  \sigma(X_0) \cup \sigma^2(X_0) \cup \sigma^3(X_0)$.
Let  $\Phi$ be the endomorphism of $(X,\sigma)$ whose right radius 3  local rule $\phi$  is 
\begin{equation}
\phi (xyzw)= \left\{
                      \begin{array}{lll}
                       1 & \mbox{ if } xyzw= 1001, \\
                       0 & \mbox{ if }
                       xyzw = 1101, \\
                        y & \mbox{otherwise}.
                      \end{array}
                    \right.
\end{equation}
Then  the local rule of $\Phi$  ``depends on the first variable", is not a power of $\sigma$,  but is $\sigma$ followed by interchanging all occurrences of $B$ and $C$. Therefore  $\Phi^2 = \sigma^2$.

\end{example}

\end{itemize}

\proof[Acknowledgements] We thank Marcus Pivato for illuminating discussions.

{\footnotesize
\bibliographystyle{alpha}
\bibliography{bibliography}
}

\end{document}